\documentclass[11pt,a4paper]{article}
\usepackage[utf8]{inputenc}
\usepackage[T1]{fontenc}
\usepackage{amsmath,amsfonts,amssymb,amsthm,mathtools}
\usepackage{mathrsfs}
\usepackage{enumitem}
\usepackage{geometry}
\usepackage{hyperref}

\numberwithin{equation}{section}

\newcommand{\IK}{\mathbb{K}}
\newcommand{\IN}{\mathbb{N}}
\newcommand{\IZ}{\mathbb{Z}}
\newcommand{\IR}{\mathbb{R}}
\newcommand{\IC}{\mathbb{C}}

\newcommand{\indic}{\mathbb{I}} 
\newcommand{\expect}{\mathbb{E}}
\newcommand{\proba}{\mathbb{P}}

\newcommand{\intd}{\text{d}}

\newcommand{\longlongrightarrow}{\xrightarrow{\hspace*{1cm}}}
\newcommand{\longlongmapsto}{\xmapsto{\hspace*{1cm}}}

\theoremstyle{plain}
\newtheorem{theorem}{Theorem}[section]
\theoremstyle{definition}
\newtheorem{definition}[theorem]{Definition}
\theoremstyle{plain}
\newtheorem{proposition}[theorem]{Proposition}
\theoremstyle{plain}

\theoremstyle{plain}
\newtheorem{lemma}[theorem]{Lemma}
\theoremstyle{remark}
\newtheorem{remark}[theorem]{Remark}

\theoremstyle{plain}

\theoremstyle{plain}

\theoremstyle{plain}

\theoremstyle{remark}

\begin{document}
\title{Frequently hypercyclic random vectors}
\date{}
\author{Kevin \textsc{Agneessens}\footnote{The author is a Research Fellow of the Fonds de la Recherche Scientifique-FNRS.}}
\maketitle

\begin{abstract}
We show that, under suitable conditions, an operator acting like a shift on some sequence space has a frequently hypercyclic random vector whose distribution is strongly mixing for the operator. This result will be applied to chaotic weighted shifts. We also apply it to every operator satisfying the Frequent Hypercyclicity Criterion, recovering a result of Murillo and Peris.
\end{abstract}

\section{Introduction}

Let $\IK$ be the set $\IR$ or $\IC$ and let $E$ be an F-space over $\IK$, that is, a completely metrizable topological vector space, and assume that $E$ is locally bounded or locally convex, see \cite{KaltonPeckRoberts1984}. Hypercyclicity is the main notion of Linear Dynamics: an operator $T : E\longrightarrow E$ that is, a continuous and linear map, is \emph{hypercyclic} if there exists a vector $x\in E$ whose orbit $\{T^n(x)\mid n\geq 0\}$ under $T$ is dense in $E$. Such a vector is called a \emph{hypercyclic vector} for $T$. An operator $T :E\longrightarrow E$ is \emph{chaotic} if it is hypercyclic and has a dense set of periodic points.

There are now plenty of known hypercyclic operators. For example, the multiples of the backward shift operator $\lambda B$ on $\ell^p$, $1\leq p<\infty$, or $c_0$, with $|\lambda|>1$, are hypercyclic, see \cite[Example 2.32]{Grosse-ErdmannManguillot2011}. Recall that the backward shift operator $B$ is defined as $B(e_n)=e_{n-1}$ for all $n\geq 1$ and $B(e_0)=0$, where $(e_n)_{n\geq 0}$ is the canonical basis. Another example is the differentiation operator $D$ on the space of entire functions $H(\IC)$ that is, $D(f)=f'$, $f\in H(\IC)$, see \cite[Example 2.35]{Grosse-ErdmannManguillot2011}. All of these operators are even chaotic.

Hypercyclicity of a vector $x$ means that its orbit visits each non-empty open set at least once, and therefore infinitely often. One can wish to quantify how often such a vector visits each non-empty open set. Recall that the \emph{lower density} of a set $A\subseteq \IN$ is the quantity $\underline{\text{dens}}(A):=\liminf_{N\to\infty}\frac{|A\cap\{0,\dots,N\}|}{N+1}$, where $\IN=\{0,1,2,\dots\}$.

\begin{definition}
Let $E$ be an F-space. An operator $T : E\longrightarrow E$ is \emph{frequently hypercyclic} if there exists some $x\in E$ such that, for every non-empty open set $U$ of $E$, the set $\{n\geq 0 \mid T^n(x)\in U\}$ has positive lower density. Such a vector is called a \emph{frequently hypercyclic vector} for $T$.
\end{definition}

The multiples of the backward shift $\lambda B$, $|\lambda|>1$, on $\ell^p$, $1\leq p<\infty$, or $c_0$ and the differentiation operator $D$ on $H(\IC)$ are all frequently hypercyclic, see \cite[Example 9.12]{Grosse-ErdmannManguillot2011} and \cite[Corollary 9.14]{Grosse-ErdmannManguillot2011}.

More information about Linear Dynamics can be found in \cite{BayartMatheron2009} and \cite{Grosse-ErdmannManguillot2011}.

Let $B_{w} : \ell^p \longrightarrow \ell^p$ be a weighted shift on $\ell^p$, $1\leq p<\infty$, where $w=(w_n)_{n\geq 1}$ is the sequence of weights. It is known that if $B_w$ is chaotic then the random vector $\sum_{n\geq 0}\frac {X_n}{w_1\dots w_n}e_n$ is almost surely frequently hypercyclic for $B_w$, where $(X_n)_{n\geq 0}$ is a sequence of independent non-constant Gaussian random variables, see \cite[Section 5.5.2]{BayartMatheron2009}, \cite[Section 7.1]{BayartMatheron2016}. Furthermore, this random vector also induces a strongly mixing Gaussian measure for $B_w$.

\begin{definition}
Let $(M,\mathcal{B},\mu)$ be a probability space.
A measurable map $T:M\longrightarrow M$ is \emph{measure-preserving} if $\mu(T^{-1}(A))=\mu(A)$ for every $A\in\mathcal{B}$.

If $T$ is measure-preserving then it is
\begin{enumerate}[label=(\roman*)]
\item \emph{ergodic} if for every $A\in\mathcal{B}$ such that $A=T^{-1}(A)$ then $\mu(A)\in\{0,1\}$,
\item \emph{strongly mixing} if $\lim_{n\to\infty}\mu(T^{-n}(A)\cap B)=\mu(A)\mu(B)$ for every $A,B\in\mathcal{B}$,
\item \emph{exact} if every $A\in\mathcal{B}$ belonging to $\bigcap_{n\geq 0}T^{-n}(\mathcal{B})$ satisfies $\mu(A)\in\{0,1\}$.
\end{enumerate}
\end{definition}

We remark that exactness implies strong mixing, and strong mixing implies ergodicity, see \cite[pp.\ 50, 87]{DajaniKalle2021}.

In \cite{Nikula2014}, Nikula proved that $\sum_{n\geq 0}\frac{X_n}{n!}e_n$ is almost surely frequently hypercyclic for the differentiation operator $D$ on the space $H(\IC)$ of entire functions, where the distribution of the independent and identically distributed (i.i.d) variables $(X_n)_{n\geq 0}$ satisfies some conditions and $(e_n)_{n\geq 0}$ is the sequence of monomials. In \cite{MouzeMunnier2014}, Mouze and Munnier relaxed the condition on the distribution. The result was also proved by Bayart and Matheron in \cite[Remark 2 after Proposition 8.1]{BayartMatheron2016} in the Gaussian case, and the random vector $\sum_{n\geq 0}\frac{X_n}{n!}e_n$ also induces a strongly mixing Gaussian measure for $D$.
As a last example, Mouze and Munnier proved in \cite[Theorem 1.3]{MouzeMunnier2021} that $\sum_{n\geq 0}X_ne_n$ is almost surely frequently hypercyclic for the so-called Taylor shift.

The aim of the paper is to generalize these results to very general chaotic weighted shifts and even to a larger class of operators. However, the sequence $(X_n)_{n\geq 0}$ might not be Gaussian.

\begin{theorem}\label{mainTheoremFHC}
Let $T : E\longrightarrow E$ be an operator and let $(u_n)_{n\in\IZ}$ be a sequence in $E$ such that $T(u_n)=u_{n-1}$ for every $n\in\IZ$ and $\emph{span}\{u_n\mid n\in\IZ\}$ is dense in $E$.
Let $X$ be a random variable with full support and let $(X_n)_{n\in\IZ}$ be a sequence of i.i.d copies of $X$.
Assume that there exists a sequence of positive numbers $(\delta_n)_{n\in\IZ}$ such that
\[
\sum_{n=-\infty}^{\infty}\proba\left(|X|\geq\delta_n\right)<\infty
\]
and the series $\sum_{n=-\infty}^{\infty}\delta_nu_n$ is unconditionally convergent in $E$.
Then the random vector
\[
v:=\sum_{n=-\infty}^{\infty}X_nu_n
\]
is almost surely well-defined and frequently hypercyclic for the operator $T$, and it induces a strongly mixing measure with full support for $T$.
If $u_n=0$ for all $n\leq -1$ then the measure is even exact for $T$.
\end{theorem}

The next section is devoted to the proof of this result. In the third section, we deduce three important special cases of the theorem: we obtain conditions under which the desired random variable $X$ exists (Theorem \ref{existenceLawFrechetSpaceGen}), or can be chosen to be subgaussian (Theorem \ref{existenceSubgaussianFHC}) or Gaussian (Theorem \ref{existenceSubgaussianType}).

In the fourth section, these results will be applied to chaotic weighted shifts on very general sequence spaces. We will also give a new proof of a result of Murillo and Peris \cite{Murillo-ArcilaPeris2013} by showing that every operator satisfying the Frequent Hypercyclicity Criterion admits a strongly mixing invariant measure with full support, where we obtain a rather explicit construction of such a measure.

Throughout the paper, if nothing else is said, let $E$ be a locally bounded or locally convex separable F-space over $\IK=\IR$ or $\IC$. If the space $E$ is complex (resp.\ real), a random variable $X$ is assumed to take complex (resp.\ real) values. Every random variable considered will be defined on a probability space $(\Omega,\mathcal{A},\proba)$.

\section{Frequent hypercyclicity}\label{frequentHypercyclicity}

The aim of this section is to prove Theorem \ref{mainTheoremFHC}.
We begin with three lemmas.

\begin{lemma}[{\cite[Lemma 6.6]{FonsecaLeoni2007}}]\label{convMeas}
Let $(F,\mathcal{A})$ and $(G,\mathcal{B}(G))$ be two measurable spaces with $G$ a metric space and $\mathcal{B}(G)$ the $\sigma$-algebra of Borel sets of $G$.
Let $(f_n)_{n\geq 0}$ be a sequence of measurable maps $f_n : F\longrightarrow G$, $n\geq 0$.
Assume that $(f_n)_{n\geq 0}$ converges pointwise to a function $f : F\longrightarrow G$. Then $f$ is measurable.
\end{lemma}

\begin{proof}
Since $\mathcal{B}(G)$ is the $\sigma$-algebra generated by the open subsets of $G$, it suffices to show that
$f^{-1}(C)\in\mathcal{A}$ for every closed subset $C$ of $G$.
So let $C\subseteq G$ be a closed subset of $G$.
It is easily verified that
\[
f^{-1}(C)=\bigcap_{k\geq 1}\bigcup_{n_0\in\IN}\bigcap_{n\geq n_0} f_n^{-1}\big(\{x\in X\mid \text{dist}(x,C)<1/k\}\big).
\]
Since $f^{-1}(C)$ can be written as countable unions and intersections of sets of $\mathcal{A}$, we have $f^{-1}(C)\in\mathcal{A}$.
\end{proof}

The proof of Lemma \ref{convDistribution} should already be known. A proof in the case of a Banach space can be found in \cite[Corollary E.1.17]{HytonenNeervenVeraarWeis2017}.

\begin{lemma}\label{convDistribution}
Let $F$ be a metric space. Let $(X_n)_{n\in\IN}$ and $(Y_n)_{n\in\IN}$ be two sequences of random variables with values in $F$ such that for every $n\in\IN$, $X_n$ and $Y_n$ have the same distribution. If $(X_n)_{n\in\IN}$ (resp.\ $(Y_n)_{n\in\IN}$) converges almost surely to $X$ (resp.\ $Y$) then the random variables $X$ and $Y$ have the same distribution.
\end{lemma}

\begin{proof}
By assumption, for every bounded continuous function $h : F\longrightarrow\IR$ and every $n\in\IN$, we have $\expect(h(X_n))=\expect(h(Y_n))$. By taking the limit when $n$ goes to $\infty$, we get $\expect(h(X))=\expect(h(Y))$.

Now, let $A\in\mathcal{B}(F)$ and $\varepsilon>0$. There exists an open set $U\subseteq F$ containing $A$ such that $\proba(X\in U\setminus A)\leq\varepsilon$ and $\proba(Y\in U\setminus A)\leq\varepsilon$ by \cite[Proposition 18.3]{Coudene2016}. For all $k\geq 0$, define the bounded and continuous function $f_k:F\longrightarrow\IR$ by $f_k(x):=\min(1,k\text{dist}(x,F\setminus U))$, $x\in F$. By the Dominated Convergence Theorem, there exists $k\geq 0$ large enough such that $|\int_{\Omega}(\indic_U(X)-f_k(X))\intd\proba|\leq\varepsilon$ and $|\int_{\Omega}(\indic_U(Y)-f_k(Y))\intd\proba|\leq\varepsilon$. Therefore, $|\proba(X\in A)-\proba(Y\in A)|\leq 4\varepsilon$. Since $\varepsilon>0$ was arbitrary, we conclude that $\proba(X\in A)=\proba(Y\in A)$ for every $A\in\mathcal{B}(F)$, and $X$ and $Y$ have the same distribution.
\end{proof}

The following lemma is well-known.

\begin{lemma}\label{convIndependence}
Let $(X_n)_{n\in\IN}$ and $(Y_n)_{n\in\IN}$ be two sequences of real random variables such that for every $n\in\IN$, $X_n$ and $Y_n$ are independent. If $(X_n)_{n\in\IN}$ (resp.\ $(Y_n)_{n\in\IN}$) converges almost surely to $X$ (resp.\ $Y$) then the random variables $X$ and $Y$ are independent.
\end{lemma}

The proof of the first result relies on the Birkhoff Ergodic Theorem, see e.g.\ \cite[Theorem 1.14]{Walters1982}.

\begin{theorem}[Birkhoff's Ergodic Theorem]
Let $(M,\mathcal{B},\mu)$ be a probability space.
Let $T:M\longrightarrow M$ be a measure-preserving and ergodic map, and let $f\in L^1(M,\mu)$. Then
\[
\lim_{N\to\infty}\frac{1}{N+1}\sum_{n=0}^Nf(T^n(x))=\int_{M}f\intd\mu\text{ $\mu$-a.s.}
\]
\end{theorem}

The next result gives conditions under which the random vector $\sum_{n=-\infty}^{\infty}X_nu_n$ is almost surely frequently hypercyclic.

\begin{proposition}\label{invErgMeasureFHC}
Let $T: E\longrightarrow E$ be an operator and let $(u_n)_{n\in\IZ}$ be a sequence in $E$ such that $T(u_n)=u_{n-1}$ for every $n\in\IZ$.
Let $(X_n)_{n\in\IZ}$ be a sequence of i.i.d random variables defined on a probability space $(\Omega,\mathcal{A},\proba)$.
Assume that the random vector
\[
v:=\sum_{n=-\infty}^{\infty}X_nu_n
\]
is almost surely well-defined and $\proba(v\in O)>0$ for every non-empty open subset $O$ of $E$. Then $v$ is almost surely frequently hypercyclic for the operator $T$ and induces a strongly mixing measure with full support for $T$.
\end{proposition}

\begin{proof}
We can assume that the series defining $v$ is convergent everywhere. Indeed, restrict the random variables $X_n$, $n\in\IZ$, to a subset of $\Omega$ of full measure on which the series defining $v$ converges. Hence we assume that the convergence is everywhere, and $v$ is measurable by Lemma \ref{convMeas}.

Define the probability measure
\[
\mu :
\mathcal{B}(E)\longlongrightarrow[0,1]
,\; A\longlongmapsto\proba(v\in A).
\]
In fact, the measure $\mu$ is the probability distribution of the random vector $v$.

First, we show that $\mu$ is $T$-invariant.
Let $A\in\mathcal{B}(E)$. By the definitions of $\mu$ and $v$ and continuity of $T$ we have
\[
\mu(T^{-1}(A))=\proba(T(v)\in A)
=\proba\left(\sum_{n=-\infty}^{\infty}X_nu_{n-1}\in A\right)\\
=\proba\left(\sum_{n=-\infty}^{\infty}X_{n+1}u_n\in A\right).
\]
Since $(X_n)_{n\in\IZ}$ is a sequence of i.i.d random variables, we have
\[
\proba\left(\sum_{n=-\infty}^{\infty}X_{n+1}u_{n}\in A\right)
=\proba\left(\sum_{n=-\infty}^{\infty}X_nu_n\in A\right)
\]
by Lemma \ref{convDistribution}. We conclude by definition of $\mu$ that $\mu(T^{-1}(A))=\proba(v\in A)=\mu(A)$. The measure $\mu$ is thus $T$-invariant.

Now we claim that $\mu$ is $T$-strongly mixing.
Let $f$ and $g$ be two bounded and continuous real-valued functions defined on $E$. We aim to show that $\lim_{n\to\infty}\int_E(f\circ T^n)g\intd\mu=\int_E f\intd\mu\int_E g\intd\mu$. Since the set of bounded continuous functions on $E$ is dense in $L^2(E,\mu)$ by \cite[Theorem 18.1]{Coudene2016}, this will imply the claim by \cite[p.\ 26]{Coudene2016}.
First, by definition of $\mu$, this is equivalent to showing that
\[
\lim_{n\to\infty}\int_{\Omega} f(T^n(v))g(v)\intd\proba=\int_{\Omega} f(v)\intd\proba\int_{\Omega} g(v)\intd\proba.
\]
Let $\varepsilon>0$.
By the Dominated Convergence Theorem and since $f$ and $g$ are continuous and bounded, there exists $N\geq 1$ such that
\begin{equation}\label{stronglyMixingIneqFHC1}
\Big\|g\left(\sum_{k=-\infty}^N X_ku_k\right)-g(v)\Big\|_{L^1(\Omega,\proba)}<\varepsilon
\end{equation}
and
\begin{equation}\label{stronglyMixingIneqFHC2}
\Big\|f\left(\sum_{k=-N}^{\infty} X_ku_k\right)-f(v)\Big\|_{L^1(\Omega,\proba)}<\varepsilon.
\end{equation}
Let $n>2N$. We have
\begin{align}\label{frequentHypercyclicitySMequality}
f(T^n(v))g(v)
&=
f(T^n(v))g(v)-
f(T^n(v))g\left(\sum_{k=-\infty}^N X_ku_k\right)\notag\\
&\quad+f(T^n(v))g\left(\sum_{k=-\infty}^N X_ku_k\right)-f\left(\sum_{k=-N}^{\infty}X_{k+n}u_k\right)g\left(\sum_{k=-\infty}^N X_ku_k\right)\notag\\
&\quad +f\left(\sum_{k=-N}^{\infty}X_{k+n}u_k\right)g\left(\sum_{k=-\infty}^N X_ku_k\right).
\end{align}
For the first two terms, using the assumption that $f$ is bounded and the inequality \eqref{stronglyMixingIneqFHC1} yields
\begin{align*}
\bigg|\int_{\Omega} f(T^n(v))g(v)\intd\proba
&-\int_{\Omega} f(T^n(v))g\left(\sum_{k=-\infty}^N X_ku_k\right)\intd\proba\bigg|\\
&\leq
\|f\|_{\infty}
\Big\|g\left(\sum_{k=-\infty}^N X_ku_k\right)-g(v)\Big\|_{L^1(\Omega,\proba)}
\leq\|f\|_{\infty}\varepsilon.
\end{align*}
Now, for the third and fourth terms, using the linearity and continuity of $T$,
\begin{align*}
\Bigg|\int_{\Omega}\Bigg[f(T^n(v))&g\left(\sum_{k=-\infty}^N X_ku_k\right)-f\left(\sum_{k=-N}^{\infty}X_{k+n}u_k\right)g\left(\sum_{k=-\infty}^N X_ku_k\right)\Bigg]\intd\proba
\Bigg|\\
&\quad\leq\|g\|_{\infty}
\left\|f\left(\sum_{k=-\infty}^{\infty}X_{k+n}u_k\right)-f\left(\sum_{k=-N}^{\infty}X_{k+n}u_k\right)
\right\|_{L^1(\Omega,\proba)}\\
&\quad =\|g\|_{\infty}
\left\|f\left(\sum_{k=-\infty}^{\infty}X_{k}u_k\right)-f\left(\sum_{k=-N}^{\infty}X_{k}u_k\right)
\right\|_{L^1(\Omega,\proba)}\\
&\quad \leq\|g\|_{\infty}\varepsilon,
\end{align*}
where we have used Lemma \ref{convDistribution} for the equality and \eqref{stronglyMixingIneqFHC2} for the last inequality.

For the last term of \eqref{frequentHypercyclicitySMequality}, since the random variables $X_n$, $n\in\IZ$, are i.i.d and $n>2N$, we have, by Lemma \ref{convIndependence} applied to $(f(\sum_{k=-N}^MX_{k+n}u_k))_{M\geq 1}$ and $(g(\sum_{k=-M}^{N}X_{k}u_k))_{M\geq 1}$ and then Lemma \ref{convDistribution} applied to $(f(\sum_{k=-N}^MX_{k+n}u_k))_{M\geq 1}$ and $(f(\sum_{k=-N}^{M}X_{k}u_k))_{M\geq 1}$,
\begin{align*}
&\int_{\Omega}f\bigg(\sum_{k=-N}^{\infty} X_{k+n}u_k\bigg)g\bigg(\sum_{k=-\infty}^N X_ku_k\bigg)\intd\proba\\
&=\int_{\Omega}f\bigg(\sum_{k=-N}^{\infty}X_{k+n}u_k\bigg)\intd\proba\int_{\Omega}g\bigg(\sum_{k=-\infty}^N X_ku_k\bigg)\intd\proba\\
&=\int_{\Omega}f\bigg(\sum_{k=-N}^{\infty}X_{k}u_k\bigg)\intd\proba\int_{\Omega}g\bigg(\sum_{k=-\infty}^N X_ku_k\bigg)\intd\proba.
\end{align*}
Therefore, using again \eqref{stronglyMixingIneqFHC1} and \eqref{stronglyMixingIneqFHC2} gives
\begin{align*}
\bigg|
&\int_{\Omega}f\left(\sum_{k=-N}^{\infty}X_{k+n}u_k\right)g\left(\sum_{k=-\infty}^N X_ku_k\right)\intd\proba
-\int_{\Omega} f(v)\intd\proba\int_{\Omega}g(v)\intd\proba\bigg|\\
&\leq
\|f\|_{\infty}
\Big\|g\left(\sum_{k=-\infty}^N X_ku_k\right)-g(v)\Big\|_{L^1(\Omega,\proba)}
+\|g\|_{\infty}
\Big\|f\left(\sum_{k=-N}^{\infty}X_ku_k\right)-f(v)\Big\|_{L^1(\Omega,\proba)}\\
&\leq\|f\|_{\infty}\varepsilon+\|g\|_{\infty}\varepsilon.
\end{align*}
We can finally conclude that
\begin{align*}
\left|\int_{\Omega}f(T^n(v))g(v)\intd\proba-\int_{\Omega}f(v)\intd\proba\int_{\Omega}g(v)\intd\proba\right|
\leq 2\|f\|_{\infty}\varepsilon+2\|g\|_{\infty}\varepsilon,
\end{align*}
and since $\varepsilon>0$ was arbitrary, $\lim_{n\to\infty}\int_{\Omega} f(T^n(v))g(v)\intd\proba=\int_{\Omega} f(v)\intd\proba\int_{\Omega} g(v)\intd\proba$.
The measure $\mu$ is thus $T$-strongly mixing.

Let $O$ be a non-empty open subset of $E$. The Birkhoff Ergodic Theorem can be applied to $T$ and $\mu$ and gives
\[
\lim_{N\to\infty}\frac{1}{N+1}\sum_{n=0}^N\mathbb{I}_{O}\circ T^n=\mu(O)\text{ $\mu$-a.s}.
\]
Let $A$ be a Borel subset of $E$ such that $\mu(A)=1$ and the previous equality holds everywhere on $A$. Then, if $B:=v^{-1}(A)\subseteq\Omega$, we have $\proba(B)=\proba(v^{-1}(A))=\mu(A)=1$ and
\[
\lim_{N\to\infty}\frac{1}{N+1}\sum_{n=0}^N\mathbb{I}_{O}\circ T^n(v)=\proba(v\in O)>0
\]
on $B$. Since $E$ is a separable F-space, we can take a countable base of open subsets of $E$ and get that almost surely, $\{n\geq 0\mid T^n(v)\in O\}$ has positive lower density for every non-empty open subset $O$ of $E$. The random vector $v$ is therefore almost surely frequently hypercyclic for the operator $T$.
\end{proof}

\begin{remark}
In fact, if $T$ admits an invariant and ergodic probability measure $\mu$ with full support then $T$ is frequently hypercyclic on $E$. This result is well-known, see e.g.\ \cite[Proposition 3.12]{BayartGrivaux2006}.
\end{remark}

If $u_n=0$ for all $n\leq-1$ in Proposition \ref{invErgMeasureFHC} then the measure induced by $v$ is even exact for $T$.

\begin{proposition}\label{mainTheoremFHCExactness}
Let $T : E\longrightarrow E$ be an operator and let $(u_n)_{n\in\IN}$ be a sequence in $E$ such that $T(u_n)=u_{n-1}$ for every $n\geq 1$ and $T(u_0)=0$.
Let $(X_n)_{n\in\IN}$ be a sequence of i.i.d random variables.
Assume that the random vector
\[
v:=\sum_{n=0}^{\infty}X_nu_n
\]
is almost surely well-defined and $\proba(v\in O)>0$ for every non-empty open subset $O$ of $E$. Then $v$ is almost surely frequently hypercyclic for the operator $T$ and induces an exact measure with full support for $T$.
\end{proposition}

\begin{proof}
Let $A\in \bigcap_{n\geq 0}T^{-n}(\mathcal{B}(E))$. We claim that $\proba(v\in A)\in\{0,1\}$.

Let $n\geq 0$, there exists $B\in\mathcal{B}(E)$ such that $A=T^{-n}(B)$. We then have, using Lemma \ref{convMeas},
\begin{align*}
\{v\in A\}&=\{T^n(v)\in B\}
=\left\{T^n\Big(\sum_{k=0}^{\infty}X_ku_k\Big)\in B\right\}\\
&=\left\{\sum_{k=n}^{\infty}X_ku_{k-n}\in B\right\}
=\left\{\sum_{k=0}^{\infty}X_{n+k}u_k\in B\right\}
\in\sigma(X_n,X_{n+1}\dots).
\end{align*}
We conclude by Kolmogorov's 0-1 law.
\end{proof}

By Propositions \ref{invErgMeasureFHC} and \ref{mainTheoremFHCExactness}, in order to prove Theorem \ref{mainTheoremFHC}, it remains to show that the series $v=\sum_{n\in\IZ}X_nu_n$ converges almost surely and the probability on $E$ induced by $v$ has full support. We first need a lemma.

\begin{lemma}[{\cite[Theorem 15.5]{Rudin1987}}]\label{seriesProdConv}
Let $(x_n)_{n\geq 1}$ be a sequence of positive numbers such that $\sum_{n\geq 1}x_n$ converges and $x_n<1$ for all $n\geq 1$. Then $\prod_{n\geq 1}(1-x_n)>0$.
\end{lemma}

The proof of Theorem \ref{mainTheoremFHC} uses some ideas from the proof of Theorem 2.3 of Mouze and Munnier \cite{MouzeMunnier2014}. In particular, the idea of the condition on the distribution of the random variable $X$ comes from that theorem.

\begin{proof}[Proof of Theorem \ref{mainTheoremFHC}]
Let $(\delta_n)_{n\in\IZ}$ be given by the assumption.
Since $\sum_{n\in\IZ}\proba\left(|X|\geq\delta_n\right)$ is a convergent series, it follows from the Borel-Cantelli lemma that
\[
\proba\bigg(\bigcup_{n_0\geq 1}\bigcap_{|n|\geq n_0}\Big\{|X_n|<\delta_n\Big\}\bigg)=1
\]
and hence, almost surely, $|X_n|<\delta_n$ for every $|n|$ large enough. Therefore, by the unconditional convergence of $\sum_{n\in\IZ}\delta_nu_n$, the random vector $v$ is almost surely well-defined, see \cite[Theorems 3.3.8 and 3.3.9]{KamthanGupta1981}.

By Propositions \ref{invErgMeasureFHC} and \ref{mainTheoremFHCExactness}, it remains to show that $\proba(v\in O)>0$ for every non-empty open subset $O$ of $E$.
It is enough to show this on a base of open subsets of $E$.

Let $\|.\|$ be an F-norm defining the topology of $E$, see \cite{KaltonPeckRoberts1984}, \cite[Definition 2.9]{Grosse-ErdmannManguillot2011}. Let $\eta>0$ and $y=\sum_{n=-d}^dy_nu_n\in E$. We shall prove that $\proba(v\in B_{\|.\|}(y,\eta))>0$, where $B_{\|.\|}(y,\eta)$ is the open ball for $\|.\|$ centred at $y$ and of radius $\eta$. Let $(\delta_n)_{n\in\IZ}$ be the sequence given by assumption. Since $\sum_{n\in\IZ}\delta_nu_n$ converges unconditionally, there exists an integer $N\geq d$ such that $\|\sum_{|n|\geq N+1}\alpha_nu_n\|<\eta/2$ whenever $|\alpha_n|\leq\delta_n$ for all $n\in\IZ$.
Define
\[
B:=\left\{
\Big\|\sum_{n=-N}^N(X_n-y_n)u_n\Big\|<\frac{\eta}{2}
\right\}
\subseteq\Omega
\]
and
\[
A:=B\cap\Big\{
|X_n|<\delta_n\text{ for all }|n|\geq N+1
\Big\},
\]
where $y_n=0$ if $d+1\leq |n|\leq N$.
By the triangle inequality we get on $A$
\begin{align*}
\|v-y\|
\leq\Big\|\sum_{n=-N}^N(X_n-y_n)u_n\Big\|
+\Big\|\sum_{|n|\geq N+1} X_nu_n\Big\|
&<\frac{\eta}{2}
+\frac{\eta}{2}
=\eta.
\end{align*}
This shows that $A\subseteq\{v\in B_{\|.\|}(y,\eta)\}$. Thus it suffices to prove that $\proba(A)>0$.
Since $(X_n)_{n\in\IZ}$ is i.i.d, we have
\[
\proba(A)=\proba(B)\prod_{|n|\geq N+1}\left(1-\proba\left(|X|\geq\delta_n\right)\right).
\]
Since $X$ has full support and $(X_n)_{n\in\IZ}$ is i.i.d, $\proba(B)>0$.
By Lemma \ref{seriesProdConv}, the product is positive since the series $\sum_{n\in\IZ}\proba(|X|\geq\delta_n)$ converges and $X$ has full support.
\end{proof}

\section{Existence of a distribution}\label{existenceDistribution}
There still remains a question in Theorem \ref{mainTheoremFHC}: does there exist a random variable $X$ satisfying the condition on the distribution? We begin with a simple proposition.

\begin{proposition}\label{existenceDistributionV1}
Let $(\delta_n)_{n\geq 0}$ be a sequence of positive numbers. Then there exist a probability space $(\Omega,\mathcal{A},\proba)$ and a random variable $X :\Omega\longrightarrow\IK$ with full support and $\sum_{n\geq 0}\proba(|X|\geq \delta_n)<\infty$ if and only if $\lim_{n\to\infty}\delta_n=\infty$.
\end{proposition}

\begin{proof}
It is easy to prove that if such a variable $X$ exists then $(\delta_n)_{n\in\IN}$ must converge to $\infty$. Indeed, assume that $(\delta_{n_k})_{k\geq 1}$ is bounded by some $M>0$ where $(n_k)_{k\geq 0}$ is increasing. Then $\sum_{k\geq 0}\proba(|X|\geq \delta_{n_k})\geq\sum_{k\geq 0} \proba(|X|\geq M)=\infty$ since $X$ has full support.

Now suppose that $\lim_{n\to\infty} \delta_n=\infty$.
By considering $\inf_{k\geq n} \delta_k$, $n\geq 0$, we can assume without loss of generality that $(\delta_n)_{n\geq 0}$ is non-decreasing. By replacing $\delta_n$ with $\delta_n-1/n$ and dropping some $\delta_n$, if necessary, we may also assume that $(\delta_n)_{n\geq 0}$ is a (strictly) increasing sequence of positive numbers.
Define $U_0=B(0,\delta_0)$ and for each $k\geq 1$, $U_k:=B(0,\delta_{k})\setminus B(0,\delta_{k-1})$ and set $m_k:=\lambda(U_k)$, $k\geq 0$, where $B(0,r)$ is the open ball in $\IK$ of center $0$ and radius $r$ and $\lambda$ is the Lebesgue measure on $\IK $.
Note that $(U_k)_{k\in\IN}$ is a partition of $\IK$.
Define
\[
\rho:=2^{-1}\sum_{k\geq 0}\frac{1}{2^km_k}\indic_{U_k}.
\]
Since
\[
\int_{\IK}\rho\intd\lambda=2^{-1}\sum_{k\geq 0}2^{-k}=1,
\]
$\rho$ is a density on $\IK$ and we consider the probability space $(\IK,\mathcal{B}(\IK),\rho\intd\lambda)$ and the random variable $X=\text{Id}_{\IK}$. It is obvious that $\int_{O}\rho\intd\lambda>0$ for every non-empty open set $O$ of $\IK$, and it remains to show that $\sum_{n\geq 0}\int_{\IK\setminus B(0,\delta_n)}\rho\intd\lambda<\infty$.

By using the definition of $\rho$, we get
\begin{align*}
\sum_{n\geq 0}\int_{\IK\setminus B(0,\delta_n)}\rho\intd \lambda
=\sum_{n\geq 0}\sum_{j\geq n+1}\int_{U_j}\rho\intd\lambda
&=2^{-1}\sum_{n\geq 0}2^{-n}=1.
\end{align*}
This shows that $\sum_{n\geq 0}\proba(|X|\geq \delta_n)$ converges.
\end{proof}

\begin{lemma}\label{positiveSequenceLemma}
Let $(e_n)_{n\geq 0}$ be a sequence in $E$.
For every sequence of scalars $(\varepsilon_n)_{n\geq 0}$ such that the series $\sum_{n\geq 0}\varepsilon_ne_n$ is unconditionally convergent, there exists a sequence of positive numbers $(\delta_n)_{n\geq 0}$ such that $\sum_{n\geq 0}\delta_ne_n$ is unconditionally convergent and $|\varepsilon_n|=o(\delta_n)$.
\end{lemma}

\begin{proof}
Let $\|.\|$ be an F-norm defining the topology of $E$.
Since $\sum_{n\geq 0}\varepsilon_ne_n$ is unconditionally convergent and by using \cite[Theorems 3.3.8 and 3.3.9]{KamthanGupta1981}, we can construct inductively an increasing sequence of positive integers $(N_k)_{k\geq 1}$ such that for every $k\geq 1$, every sequence $(\alpha_n)_{n\geq 0}$ of scalars with $\sup_{n\geq 0}|\alpha_n|\leq 1$ and every finite set $F\subseteq\IN$ with $\min F>N_k$, one has $\|\sum_{n\in F}\alpha_n\varepsilon_ne_n\|\leq 1/k^2$. For each $n>N_1$, there exists a unique $k\geq 1$ such that $N_k<n\leq N_{k+1}$, and we set $\delta_n=k^{1/2}|\varepsilon_n|$.
We then have for every $1\leq k<k'$ and every finite set $F\subseteq\IN$ with $N_k<\min F\leq \max F\leq N_{k'}$,
\[
\Big\|\sum_{n\in F}\delta_ne_n\Big\|
=\Big\|\sum_{s=k}^{k'-1}\sum_{n=N_s+1,\:n\in F}^{N_{s+1}}\delta_ne_n\Big\|
\leq\sum_{s=k}^{k'-1}(1+s^{1/2})s^{-2},
\]
where we have used that an F-norm satisfies that $\|\alpha x\|\leq(1+|\alpha|)\|x\|$ for any scalar $\alpha$ and $x\in E$, see \cite[p.\ 35]{Grosse-ErdmannManguillot2011}. Since $\sum_{s\geq 1}(1+s^{1/2})s^{-2}$ is convergent, the series $\sum_{n\geq 0}\delta_ne_n$ is unconditionally convergent too.
In addition, we have that $|\varepsilon_n|=o(\delta_n)$ as $n$ goes to $\infty$.
\end{proof}

We immediately deduce the main result of this subsection, which gives conditions for an operator to have a frequently hypercyclic random vector.

\begin{theorem}\label{existenceLawFrechetSpaceGen}
Let $T$ be an operator on $E$ and let $(u_n)_{n\in\IZ}$ be a sequence in $E$.
Assume that $T(u_n)=u_{n-1}$ for every $n\in\IZ$, the series
$\sum_{n\in\IZ}u_n$ is unconditionally convergent and $\text{\emph{span}}\{u_n\mid n\in\IZ\}$ is dense in $E$.
Then there exists a random variable X with full support such that the random vector
\[
\sum_{n=-\infty}^{\infty}X_nu_n
\]
is almost surely well-defined and frequently hypercyclic for the operator $T$, and it induces a strongly mixing measure with full support for $T$, where $(X_n)_{n\in\IZ}$ is a sequence of i.i.d copies of $X$.
If $u_n=0$ for all $n\leq -1$ then the measure is even exact for $T$.
\end{theorem}

\begin{proof}
Let $(\delta_n)_{n\in\IZ}$ be the sequence of positive numbers obtained by applying Lemma \ref{positiveSequenceLemma} to $\sum_{n\geq 0}u_n$ and $\sum_{n\leq -1}u_n$. Then $\lim_{n\to\infty}\delta_n=\infty$ and $\lim_{n\to-\infty}\delta_n=\infty$. The result follows by applying Proposition \ref{existenceDistributionV1} to $(\min(\delta_n,\delta_{-n}))_{n\geq 0}$ in order to obtain the existence of the random variable $X$ with full support such that $\sum_{n\in\IZ}\proba(|X|\geq\delta_n)<\infty$, and then by using Theorem \ref{mainTheoremFHC}.
\end{proof}

\begin{remark}\label{existenceDistributionRemarkCondExistence}
We are mostly only interested in the existence of a random variable $X$ as given in Theorem \ref{existenceLawFrechetSpaceGen}. For a more precise information on which random variable can be employed, one has to go back to Theorem \ref{mainTheoremFHC}.
\end{remark}

Gaussian distributions are probably the most well-known probability distributions on infinite-dimensional spaces. Thus one may ask when the random variable $X$ can be subgaussian. We present two ways to achieve this.

\begin{definition}\label{definitionSubgaussian}
A random variable $X$ is \emph{subgaussian} if there exists $K>0$ and $\tau>0$ such that $\proba(|X|>t)\leq Ke^{-t^2/\tau^2}$ for every $t\geq 0$.

A sequence of random variables $(X_n)_{n\geq 0}$ is \emph{subgaussian} if each $X_n$, $n\geq 0$, is subgaussian with the same constants $\tau$ and $K$.
\end{definition}

This definition of a subgaussian variable and an equivalent one can be found in \cite[pp.\ 4-5]{Kahane1960}. A Gaussian variable is of course subgaussian, see \cite[Chapitre 8, Proposition I.1]{LiQueffelec2004}.

One could call $(X_n)_{n\geq 0}$ a \emph{uniformly subgaussian} sequence to stress the fact that the constants $\tau$ and $K$ are the same for each random variable of the sequence.

The first method to allow $X$ to be subgaussian is by assuming the unconditional convergence of the series $\sum_{n\in\IZ^*}\sqrt{\log(|n|)}u_n$, where $\IZ^{*}=\IZ\setminus\{0\}$. This assumption guarantees the almost sure convergence of the random series $\sum_{n=-\infty}^{\infty}X_nu_n$, where $(u_n)_{n\in\IZ}$ is a sequence in $E$.

\begin{lemma}\label{existenceSubgaussianFHCwellDefined}
Let $(u_n)_{n\in\IZ}$ be a sequence of vectors of $E$.
Assume that
$\sum_{n\in\IZ^*}\sqrt{\log(|n|)}u_n$
is unconditionally convergent.
Then for every subgaussian sequence $(X_n)_{n\in\IZ}$, the random vector
\[
\sum_{n=-\infty}^{\infty}X_nu_n
\]
is almost surely well-defined.
In particular, the result holds for every non constant Gaussian variable. 
\end{lemma}

\begin{proof}
Let $c>0$. We have by definition of a subgaussian sequence that
\[
\sum_{n\in\IZ^*}\proba\left(|X_n|\geq c\sqrt{\log(|n|)}\right)
\leq K\sum_{n\in\IZ^*}e^{-c^2\log(|n|)/\tau^2}
=\sum_{n\in\IZ^*}\frac{K}{|n|^{c^2/\tau^2}},
\]
for some constants $K>0$ and $\tau>0$.
If $c^2>\tau^2$ then $\sum_{n\in\IZ^*}\proba(|X|\geq c\sqrt{\log(|n|)})$ converges. It follows from the Borel-Cantelli lemma that
\[
\proba\bigg(\bigcup_{n_0\geq 1}\bigcap_{|n|\geq n_0}\Big\{|X_n|<c\sqrt{\log(|n|)}\Big\}\bigg)=1
\]
and hence, almost surely, $|X_n|<c\sqrt{\log(|n|)}$ for every $|n|$ large enough. Therefore, by the unconditional convergence of $\sum_{n\in\IZ^*}\sqrt{\log(|n|)}u_n$, the series $\sum_{n\in\IZ}X_nu_n$ is almost surely convergent. Furthermore, it is also measurable by Lemma \ref{convMeas}.
\end{proof}

\begin{theorem}\label{existenceSubgaussianFHC}
Let $T$ be an operator on $E$ and let $(u_n)_{n\in\IZ}$ be a sequence in $E$.
Assume that $T(u_n)=u_{n-1}$ for every $n\in\IZ$, $\text{\emph{span}}\{u_n\mid n\in\IZ\}$ is dense in $E$ and 
assume that the series
$
\sum_{n\in\IZ^*}\sqrt{\log(|n|)}u_n
$
is unconditionally convergent.
Then for every subgaussian random variable $X$ with full support, the random vector
\[
\sum_{n=-\infty}^{\infty}X_nu_n
\]
is almost surely well-defined and frequently hypercyclic for the operator $T$, and it induces a strongly mixing measure with full support for $T$, where $(X_n)_{n\in\IZ}$ is a sequence of i.i.d copies of $X$.
If $u_n=0$ for all $n\leq -1$, then the measure is even exact for $T$.
In particular, the result holds for every non constant Gaussian variable.
\end{theorem}

\begin{proof}
As in the proof of Lemma \ref{existenceSubgaussianFHCwellDefined}, we have that there is some $c>0$ such that
\[
\sum_{n\in\IZ^*}\proba\left(|X|\geq c\sqrt{\log(|n|)}\right)<\infty.
\]
The result then follows by Theorem \ref{mainTheoremFHC}.
\end{proof}

The following result uses a different assumption than Theorem \ref{mainTheoremFHC}, in the case where $E$ is a Banach space. 
Recall the definition of type.
A \emph{standard Gaussian variable} is a Gaussian random variable of mean $0$ and variance $1$.
\begin{definition}
Let $E$ be a Banach space and $1\leq p\leq 2$. Then $E$ has \emph{type} $p$ if there exists $C>0$ such that for every $x_1,\dots,x_n\in E$, $n\geq 1$,
\[
\bigg\|\sum_{k=1}^nX_kx_k\bigg\|_{L^1(\Omega,\proba)}
\leq C\left(\sum_{k=1}^n\|x_k\|^p\right)^{1/p},
\]
where $(X_k)_{k=1}^n$ is a sequence of independent standard Gaussian variables.
\end{definition}

The definition is usually expressed with a Rademacher sequence
and sometimes in terms of the $L^2(\Omega,\proba)$-norm. But by \cite[Proposition 7.1.18]{HytonenNeervenVeraarWeis2017} and the Kahane-Khintchine inequalities \cite[Theorem 6.2.6]{HytonenNeervenVeraarWeis2017}, this leads to the same definition.

\begin{theorem}\label{existenceSubgaussianType}
Assume that $E$ is a Banach space of type $1\leq p \leq 2$.
Let $T$ be an operator on $E$ and let $(u_n)_{n\in\IZ}$ be a sequence in $E$.
Assume that $T(u_n)=u_{n-1}$ for every $n\in\IZ$ and $\text{\emph{span}}\{u_n\mid n\in\IZ\}$ is dense in $E$.
Assume that the series $\sum_{n=-\infty}^{\infty}\|u_n\|^p$ converges.
Let $X$ be a standard Gaussian random variable with full support and let $(X_n)_{n\in\IZ}$ be a sequence of i.i.d copies of $X$.
Then the random vector 
\[
v:=\sum_{n=-\infty}^{\infty}X_nu_n
\]
is almost surely well-defined and frequently hypercyclic for the operator $T$, and it induces a strongly mixing measure with full support for $T$.
If $u_n=0$ for all $n\leq -1$, then the measure is even exact for $T$.
\end{theorem}

\begin{proof}
Since $E$ has type $p$, we have, for every $M\geq N$
\begin{equation}\label{existenceSubgaussianTypeIneq}
\expect\left(\Bigg\|\sum_{n=N}^M X_nu_n\Bigg\|\right)\leq C_p \left(\sum_{n=N}^{M}\|u_n\|^p\right)^{1/p}
\end{equation}
where $C_p>0$ is some constant depending only on $p$.
Therefore, the random series $\sum_{n=-\infty}^{\infty}X_nu_n$ converges in $L^1(\Omega;E)$, and since $(X_n)_{n\in\IZ}$ is a standard Gaussian sequence of independent random variables, $v$ is almost surely well-defined by \cite[Corollary 6.4.4]{HytonenNeervenVeraarWeis2017}.

By Proposition \ref{invErgMeasureFHC}, it remains to show that $\proba(v\in O)>0$ for every non-empty open subset $O$ of $E$.
It is enough to show this on a base of open subsets of $E$.

So let $\eta>0$ and $y=\sum_{n=-d}^dy_nu_n\in E$. We shall prove that $\proba(v\in B_{\|.\|}(y,\eta))>0$ where $B_{\|.\|}(y,\eta)$ is the open ball centred at $y$ and of radius $\eta$. Let $N\geq d$ be an integer.
Define
\[
B:=\left\{
\bigg\|\sum_{n=-N}^N(X_n-y_n)u_n\bigg\|<\frac{\eta}{2}
\right\},
\text{ }
C:=\left\{
\bigg\|\sum_{|n|\geq N+1}X_nu_n\bigg\|<\frac{\eta}{2}
\right\},
\]
where $y_n:=0$ if $d<|n|\leq N$, and let $A:=B\cap C$. By the triangle inequality we get on $A$
\[
\|v-y\|
\leq\bigg\|\sum_{n=-N}^N(X_n-y_n)u_n\bigg\|
+\bigg\|\sum_{|n|\geq N+1}X_nu_n\bigg\|
<\frac{\eta}{2}+\frac{\eta}{2}=\eta.
\]
This shows that $A\subseteq\{v\in B_{\|.\|}(y,\eta)\}$. Thus it suffices to prove that $\proba(A)>0$. Since $(X_n)_{n\in\IZ}$ is i.i.d, we have by Lemma \ref{convIndependence} that
\[
\proba(A)=\proba(B)\proba(C).
\]
Since $X$ has full support and $(X_n)_{n\in\IZ}$ is i.i.d, we have $\proba(B)>0$. The last step is to show that $\proba(C)>0$. The Markov inequality yields
\[
1-\proba(C)
=\proba\left(\Bigg\|\sum_{|n|\geq N+1}X_nu_n\Bigg\|\geq \eta/2\right)
\leq(\eta/2)^{-1}\expect\left(\Bigg\|\sum_{|n|\geq N+1}X_nu_n\Bigg\|\right).
\]
It follows from \eqref{existenceSubgaussianTypeIneq} that if we take $N\geq d$ large enough then $1-\proba(C)<1$, i.e.\ $\proba(C)>0$.
\end{proof}

\section{Applications}

\subsection{Weighted shifts}

We list the applications of Theorem \ref{existenceLawFrechetSpaceGen} and Theorem \ref{existenceSubgaussianFHC} to unilateral and bilateral weighted shifts.

A \emph{sequence space over $\IN$ (resp.\ over $\IZ$)} $E$ is a subspace of $\IK^{\IN}$ (resp.\ $\IK^{\IZ}$) such that convergence in $E$ implies convergence in $\IK^{\IN}$ (resp.\ $\IK^{\IZ}$).
The vectors $e_n=(\dots,0,1,0,\dots)$ where $1$ lies at the $n$-th coordinate, $n\geq 0$ (resp.\ $n\in\IZ$), are called the \emph{canonical unit sequences}.
Let $E$ be a sequence space over $\IN$ (resp.\ over $\IZ$) such that the canonical unit sequences span a dense subspace. A \emph{unilateral (resp.\ bilateral) weighted shift} $T:E\longrightarrow E$ is an operator such that $T(e_n)=w_ne_{n-1}$ for all $n\geq 1$, and $T(e_0)=0$ (resp.\ $T(e_n)=w_ne_{n-1}$ for all $n\in\IZ$), where $(w_n)_{n}$ is a sequence of nonzero scalars called the \emph{weight sequence}.

Let $T$ be a weighted shift with weight sequence $(w_n)_{n}$. If $T$ is a unilateral weighted shift, define $\beta_n:=w_1\dots w_n$ if $n\geq 1$, and $\beta_0:=1$. If $T$ is a bilateral weighted shift, define $\beta_n:=w_1\dots w_n$ if $n\geq 1$, $\beta_n:=(\prod_{k=-n+1}^{0}w_k)^{-1}$ if $n\leq -1$, and $\beta_0:=1$.

In the first two results, let $E$ be a locally bounded or locally convex F-sequence space over $\IN$ in which $\text{span}\{e_n\mid n\in\IN\}$ is dense. We will then apply the results of Section \ref{existenceDistribution} to $u_n=\frac{e_n}{\beta_n}$ for $n\geq 0$ and $u_n=0$ for $n\leq -1$.

\begin{theorem}\label{existenceLawFrechetSpaceUnilateralWeightedShift}
Let $T:E\longrightarrow E$ be a weighted shift with sequence of weights $(w_n)_{n\geq 1}$.
\begin{enumerate}[label=(\roman*)]
\item Assume that the series
$\sum_{n\in\IN}\frac{e_n}{\beta_n}$ is unconditionally convergent.
Then there exists a random variable X with full support such that the random vector
\[
\sum_{n=0}^{\infty}\frac{X_n}{\beta_n}e_n
\]
is almost surely well-defined and frequently hypercyclic for the operator $T$, and it induces an exact measure with full support for $T$, where $(X_n)_{n\geq 0}$ is a sequence of i.i.d copies of $X$.
\item If the series
$
\sum_{n\geq 1}\frac{\sqrt{\log(n)}}{\beta_n}e_n
$
is unconditionally convergent then $X$ can be any subgaussian random variable with full support. In particular, the result holds for every non constant Gaussian variable.
\end{enumerate}
\end{theorem}

This generalizes the qualitative parts of \cite[Theorem 2.3]{MouzeMunnier2014} and \cite[Theorem 1]{Nikula2014}; their quantitative parts are contained in Theorem \ref{mainTheoremFHC}.

We next consider a locally bounded or locally convex F-sequence space $E$ over $\IZ$ in which $\text{span}\{e_n\mid n\in\IZ\}$ is dense. We then apply the results of Section \ref{existenceDistribution} to $u_n=\frac{e_n}{\beta_n}$, $n\in\IZ$.

\begin{theorem}\label{existenceLawFrechetSpaceBilateralWeightedShift}
Let $T:E\longrightarrow  E$ be a bilateral weighted shift with sequence of weights $(w_n)_{n\in\IZ}$.
\begin{enumerate}[label=(\roman*)]
\item Assume that the series
$\sum_{n\in\IZ}\frac{e_n}{\beta_n}$ is unconditionally convergent.
Then there exists a random variable X with full support such that the random vector
\[
\sum_{n=-\infty}^{\infty}\frac{X_n}{\beta_n}e_n
\]
is almost surely well-defined and frequently hypercyclic for the operator $T$, and it induces a strongly mixing measure with full support for $T$, where $(X_n)_{n\in\IZ}$ is a sequence of i.i.d copies of $X$.
\item If the series
$
\sum_{n\in\IZ^*}\frac{\sqrt{\log(|n|)}}{\beta_n}e_n
$
is unconditionally convergent then $X$ can be any subgaussian random variable with full support. In particular, the result holds for every non constant Gaussian variable.
\end{enumerate}
\end{theorem}

Theorems \ref{existenceLawFrechetSpaceUnilateralWeightedShift} and \ref{existenceLawFrechetSpaceBilateralWeightedShift} apply, in particular, to any chaotic unilateral or bilateral weighted shift on an F-sequence space in which $(e_n)_n$ is an unconditional basis, see \cite[Theorems 4.8, 4.13]{Grosse-ErdmannManguillot2011}.
The existence of an exact or strongly mixing measure with full support has already been proved in \cite[Corollary 2 and Remark 3]{Murillo-ArcilaPeris2013}. A different approach has also led to the existence of a strongly mixing measure in \cite[Theorem 1]{LopesMessaoudiStadlbauerVargas2021} for a class of weighted shifts on $c_0(\IN)$ or $\ell^p(\IN)$, $1\leq p<\infty$.

\begin{remark}
The bilateral weighted shift on $\ell^2(\IZ)$ with weights $w_n=2$, $n\geq 1$, and $w_n=1/2$, $n\leq 0$, is invertible and satisfies the assumptions of Theorem \ref{existenceLawFrechetSpaceBilateralWeightedShift}. On the other hand, no invertible measure preserving transformation can be exact, see \cite[p.\ 86]{DajaniKalle2021}. Thus the measure induced by the vector $v$ in Theorem \ref{mainTheoremFHC} cannot be exact for all operators $T$.
\end{remark}

It might be an interesting fact that on the space $H(\IC)$ of entire functions or the space $H(D(0,R))$ of holomorphic functions on $D(0,R):=\{z\in\IC\mid |z|<R\}$, every chaotic weighted shift satisfies the assumption of the second assertion of Theorem \ref{existenceLawFrechetSpaceUnilateralWeightedShift}.

\begin{theorem}\label{existenceSubgaussianHCHD}
On the space $E=H(\IC)$ or $H(D(0,R))$ with $R>0$, let $T:E\longrightarrow E$ be a chaotic weighted shift with sequence of weights $(w_n)_{n\geq 1}$.
Then for every subgaussian random variable $X$ with full support 
the random function
\[
\sum_{n=0}^{\infty}\frac{X_n}{\beta_n}z^n
\]
is almost surely well-defined and frequently hypercyclic for the operator $T$, and it induces an exact measure with full support for $T$, where $(X_n)_{n\in\IZ}$ is a sequence of i.i.d copies of $X$.
In particular, the result holds for every non constant Gaussian variable.
\end{theorem}

\begin{proof}
By Theorem \ref{existenceLawFrechetSpaceUnilateralWeightedShift}, it suffices to show that if $T$ is chaotic on $E$ then the series $\sum_{n\geq 1}\frac{\sqrt{\log(n)}}{\beta_n}e_n$ is unconditionally convergent in $E$.

On $H(\IC)$, $T$ is chaotic if and only if $\lim_{n\to\infty} |\beta_n|^{1/n}=\infty$, see \cite[Example 4.9(b)]{Grosse-ErdmannManguillot2011}. Therefore, for every $r\geq 1$ and $0<\rho<1$, there exists $n_0\geq 1$ such that for every $n\geq n_0$, $r^n\sqrt{\log(n)}/|\beta_n|\leq \rho^n$.

On $H(D(0,R))$, $T$ is chaotic if and only if $\limsup_{n\to\infty}|\beta_n|^{-1/n}\leq 1/R$, see \cite[Theorem 4.8]{Grosse-ErdmannManguillot2011}. Let $0<r<R$ and $0<\rho<1$ such that $r<\rho R$, there exists $n_0\geq 1$ such that for every $n\geq n_0$, $\log(n)^{1/(2n)}|\beta_n|^{-1/n}\leq\rho/r$ and hence $r^n\sqrt{\log(n)}/|\beta_n|\leq \rho^n$.
\end{proof}

For the differentiation operator on $H(\IC)$, this result was proved in the Gaussian case in \cite[Remark 2 after Proposition 8.1]{BayartMatheron2016}. For the Taylor shift on $H(D(0,1))$, which is given by the weights $w_n=1$, $n\geq 1$, the frequent hypercyclicity of the random function was proved in the Gaussian case in \cite[Theorem 1.3]{MouzeMunnier2021}.

One can ask the same question about the spaces $\ell^p$, $1\leq p< \infty$. In fact, it is already known that $\sum_{n=0}^{\infty}\frac{X_n}{\beta_n}e_n$ is almost surely well-defined and frequently hypercyclic on those spaces if the random variables $X_n$, $n\geq 0$, are Gaussian and the weighted shift is chaotic as said in the introduction, see \cite[Section 5.5.2]{BayartMatheron2009}, \cite[Section 7.1]{BayartMatheron2016}. However, the second assertion of Theorem \ref{existenceLawFrechetSpaceUnilateralWeightedShift} cannot be applied to every chaotic weighted shift defined on $\ell^p$, $1\leq p<\infty$. Indeed, consider the sequence $(\beta_n)_{n\geq 1}:=(\log(n)^{1/2+1/p}n^{1/p})_{n\geq 1}$. Then $\sum_{n\geq 0}\sqrt{\log(n)}/\beta_ne_n$ is not in $\ell^p $ but the weighted shift associated to $(\beta_n)_{n\geq 1}$ is chaotic. Note that Theorem \ref{existenceSubgaussianType} can be applied to any chaotic weighted shift on $\ell^p$, $1\leq p\leq 2$.

In their article \cite{MouzeMunnier2014}, Mouze and Munnier have also studied some polynomials of a frequently hypercyclic weighted shift on $\ell^p$, $1\leq p<\infty$. But their Lemma 4.1 says that certain polynomials of a weighted shift can be seen as a shift with respect to another basis. 
The proof of this lemma shows the following.

\begin{lemma}[{\cite[Lemma 4.1]{MouzeMunnier2014}}]\label{polynomialShiftBasis}
Let $T:\IK^{\IN}\longrightarrow\IK^{\IN}$ be a weighted shift. Let $P(z)=\sum_{k=1}^d a_kz^k$ be a polynomial with $a_1\neq 0$. Then there exist vectors $u_n=\sum_{j=0}^n\beta_{j,n}e_j$ such that $(u_n)_{n\geq 0}$ is an algebraic basis of $\IK^{\IN}$ and $P(T)(u_n)=u_{n-1}$ for every $n\geq 1$.
\end{lemma}

In fact, this result implies that $\text{span}\{e_n\mid n\in\IN\}=\text{span}\{u_n\mid n\in\IN\}$.
Note also that $P(T)(u_0)=0$.
Therefore, together with Theorem \ref{existenceLawFrechetSpaceGen}, we deduce the following result.

\begin{theorem}\label{existenceLawFrechetSpacePoly}
Let $E$ be a locally bounded or locally convex F-sequence space in which $\text{\emph{span}}\{e_n\mid n\in\IN\}$ is dense. Let $T: E\longrightarrow E$ be a weighted shift and $P(z)=\sum_{k=1}^d a_kz^k$ be a polynomial with $a_1\neq 0$.
Assume that
the series
$\sum_{n\geq 0}u_n$ is unconditionally convergent, where $(u_n)_{n\geq 0}$ is given by Lemma \ref{polynomialShiftBasis}.
Then there exists a random variable X with full support such that the random vector
\[
v:=\sum_{n=0}^{\infty}X_nu_n
\]
is almost surely well-defined and frequently hypercyclic for the operator $P(T)$, and it induces an exact measure with full support for $P(T)$, where $(X_n)_{n\geq 0}$ is a sequence of i.i.d copies of $X$.
\end{theorem}

This result improves and generalizes the qualitative part of \cite[Theorem 4.3]{MouzeMunnier2014}; its quantitative part is contained in Theorem \ref{mainTheoremFHC}.

\subsection{Operators satisfying the Frequent Hypercyclicity Criterion}

First recall the Frequent Hypercyclicity Criterion, see \cite[Theorem 2.1]{BonillaGrosse-Erdmann2007}.

\begin{theorem}[Frequent Hypercyclicity Criterion]\label{prerequisitesFHC}
Let $T$ be an operator on a separable $F$-space $E$. Assume that there exists a dense subset $E_0$ of $E$ and a map $S : E_0\longrightarrow E_0$ such that for any $x\in E_0$, the following conditions hold:
\begin{enumerate}[label=(\roman*)]
\item $\sum_{n\geq 0}T^n(x)$ is unconditionally convergent,
\item $\sum_{n\geq 0}S^n(x)$ is unconditionally convergent,
\item $TS(x)=x$.
\end{enumerate}
Then $T$ is frequently hypercyclic.
\end{theorem}

In \cite[Theorem 1]{Murillo-ArcilaPeris2013}, Murillo and Peris proved that every operator satisfying the Frequent Hypercyclicity Criterion has a strongly mixing invariant measure with full support. They used the Bernoulli shift on a subset of $\IN^{\IZ}$ to construct such a measure. In \cite[Proposition 8.1]{BayartMatheron2016}, it is even shown that such operators admit a strongly mixing Gaussian measure.
We will show here the existence of a strongly mixing measure with full support as the distribution of some random vector $\sum_{n\in\IZ}X_nu_n$.
We will need the next lemma. Its proof is contained in the proof of \cite[Lemma 3.2]{Grivaux2006}. In this subsection, $E$ will be again a locally bounded or locally convex separable F-space.

Recall that a vector $x\in E$ is \emph{supercyclic} for an operator $T:E \longrightarrow E$ if the set $\{\lambda T^n(x)\mid n\geq 0, \lambda\in\IK\}$ is dense in $E$.

\begin{lemma}\label{existenceLawFrechetSpaceFHCLemma}
Let $T$ be an operator on $E$ satisfying the Frequent Hypercyclicity Criterion and let $S$ and $E_0$ be respectively the map and dense set given by that criterion. Let $(a_k)_{k\geq 1}$ be a sequence of non-zero scalars. If $(x_k)_{k\geq 1}$ is a dense sequence in $E_0$ then there exists an increasing sequence $(n_k)_{k\geq 1}$ of positive integers such that the vector $x:=\sum_{k\geq 1}a_kS^{n_k}(x_k)$ is well-defined and supercyclic for $T$.
\end{lemma}

\begin{proof}
By conditions \emph{(i)} and \emph{(ii)} of the Frequent Hypercyclicity Criterion, we know that $(T^n(x))_{n\geq 0}$ and $(S^n(x))_{n\geq 0}$ converge to $0$ for every $x\in E_0$. Therefore, together with \emph{(iii)}, one can construct by induction an increasing sequence of positive integers $(n_k)_{k\geq 1}$ such that
$\|a_k S^{n_k}(x_k)\|\leq\frac{1}{2^{k}}$ for every $k\geq 1$,
and
\[
\bigg\|\frac{1}{a_l}T^{n_l}\Big(\displaystyle\sum_{j=1}^k a_jS^{n_j}(x_j)\Big)-x_l\bigg\|< \frac{1}{2^l}
\text{ for every }1\leq l\leq k,
\]
where $\|.\|$ is an F-norm defining the topology of $E$.
The first condition tells us that $(\sum_{k=1}^s a_kS^{n_k}(x_k))_{s\geq 1}$ is Cauchy in $E$, hence converges. The second condition tells us that the vector $x:=\sum_{k\geq 1}a_kS^{n_k}(x_k)$ is supercyclic for $T$.
\end{proof}

\begin{theorem}\label{existenceLawFrechetSpaceFHC}
Let $T$ be an operator on $E$ satisfying the Frequent Hypercyclicity Criterion.
Then there exists a supercyclic vector $x$ for $T$, a sequence $(u_n)_{n\geq 0}$ in $E$ with $u_0=x$ and $T(u_n)=u_{n-1}$ for every $n\geq 1$, and a random variable X with full support such that the random vector
\[
v:=\sum_{n=0}^{\infty}X_nT^n(x)+\sum_{n=1}^{\infty}X_{-n}u_n
\]
is almost surely well-defined and frequently hypercyclic for the operator $T$, and it induces a strongly mixing measure with full support for $T$, where $(X_n)_{n\in\IZ}$ is a sequence of i.i.d copies of $X$.
\end{theorem}

\begin{proof}
Let $S$ be the map and $E_0$ the dense set given by the Frequent Hypercyclicity Criterion and
let $\|.\|$ be an F-norm defining the topology of $E$.
Let $(x_k)_{k\geq 1}$ be a dense sequence in $E_0$.

For each $k\geq 1$, choose a real number $0<a_k<1$ such that
\begin{equation}\label{existenceLawFrechetSpaceFHCineq1}
\sup_{F\subseteq\IN,\text{ }F\text{ finite}}\bigg\|\sum_{n\in F}a_kS^{n}(x_k)\bigg\|\leq\frac{1}{2^k}
\end{equation}
and
\begin{equation}\label{existenceLawFrechetSpaceFHCineq2}
\sup_{F\subseteq\IN,\text{ }F\text{ finite}}\bigg\|\sum_{n\in F}a_kT^{n}(x_k)\bigg\|\leq\frac{1}{2^k}.
\end{equation}
This is possible by \emph{(i)} and \emph{(ii)} of the Frequent Hypercyclicity Criterion.
Indeed, by unconditional convergence and \cite[Theorems 3.3.8 and 3.3.9]{KamthanGupta1981}, there exists $N\geq 1$ such that
$
\|\sum_{n\in F}a_kS^{n}(x_k)\|\leq 2^{-k-1}
$
whenever $\min F\geq N$ and $|a_k|\leq 1$, and by continuity one can choose $a_k>0$ small enough to get
$
\|\sum_{n\in F}a_kS^{n}(x_k)\|\leq 2^{-k-1}
$
whenever $\max F\leq N$. The same arguments hold for the second inequality.

Now let $(n_k)_{k\geq 1}$ be the sequence given by Lemma \ref{existenceLawFrechetSpaceFHCLemma} and define the vector $x:=\sum_{k\geq 1}a_kS^{n_k}(x_k)$.
If $n\geq 0$, by the triangle inequality we have by \eqref{existenceLawFrechetSpaceFHCineq1}, for every $M\geq N\geq 1$,
\[
\Bigg\|\sum_{k=N}^Ma_kS^{n_k+n}(x_k)\Bigg\|
\leq\sum_{k=N}^M\|a_kS^{n_k+n}(x_k)\|
\leq\sum_{k=N}^M\frac{1}{2^k}
\]
and hence
\[
u_n:=\sum_{k\geq 1}a_kS^{n_k+n}x_k,
\]
$n\geq 0$, is well-defined where $u_0=x$. We also set $u_n=T^{-n}(x)$, $n\leq -1$. It is then easy to check that $T(u_n)=u_{n-1}$ for every $n\in\IZ$.
We will apply Theorem \ref{existenceLawFrechetSpaceGen} to $(u_n)_{n\in\IZ}$. For the statement of the theorem, we then replace $(X_n)_{n\in\IZ}$ by $(X_{-n})_{n\in\IZ}$. Note that $\text{span}\{u_n\mid n\in\IZ\}$ is dense in $E$ since $x$ is supercyclic for $T$.

Thus it remains to show that $\sum_{n\in\IZ}u_n$ is unconditionally convergent. Let $\varepsilon>0$ and let $k_0\geq 1$ be such that $\sum_{k\geq k_0+1}2^{1-k}\leq\varepsilon$. For each $k\geq 1$, by \emph{(i)} and \emph{(ii)} of the Frequent Hypercyclicity Criterion, there exists $N_k\geq 1$ such that
\[
\bigg\|\sum_{n\in F}a_kT^{n}(x_k)\bigg\|<\frac{\varepsilon}{k_0}
\quad\text{and}\quad
\bigg\|\sum_{n\in F}a_kS^{n}(x_k)\bigg\|<\frac{\varepsilon}{k_0}
\]
for every finite set $F\subseteq\IN$ with $\min F\geq N_k$.
Let $F\subseteq\IN$ be a finite subset with $\min F \geq \max_{1\leq k\leq k_0}(N_k+n_k)$. We have
\begin{align*}
\sum_{n\in F}u_{-n}
&=\sum_{n\in F}\sum_{k\geq 1}a_kT^nS^{n_k}(x_k)
=\sum_{k\geq 1}\sum_{n\in F}a_kT^nS^{n_k}(x_k)\\
&=\sum_{k=1}^{k_0}\sum_{n\in F}a_kT^{n-n_k}(x_k)
+\sum_{k\geq k_0+1}\sum_{n\in F}a_kT^nS^{n_k}(x_k).
\end{align*}
The first term is smaller than $\varepsilon$ with respect to $\|.\|$ since $\min F\geq N_k+n_k$ for each $1\leq k\leq k_0$. The triangle inequality, inequalities \eqref{existenceLawFrechetSpaceFHCineq1} and \eqref{existenceLawFrechetSpaceFHCineq2} and condition \emph{(iii)} of the Frequent Hypercyclicity Criterion yield
\begin{align*}
\bigg\|\sum_{k\geq k_0+1}&\sum_{n\in F}a_kT^nS^{n_k}(x_k)\bigg\|
\leq\sum_{k\geq k_0+1}\bigg\|\sum_{n\in F}a_kT^nS^{n_k}(x_k)\bigg\|	\\
&\leq\sum_{k\geq k_0+1}\bigg(\Big\|\sum_{n\in F, n<n_k}a_kS^{n_k-n}(x_k)\Big\|
+\Big\|\sum_{n\in F, n\geq n_k}a_kT^{n-n_k}(x_k)\Big\|\bigg)\\
&\leq \sum_{k\geq k_0+1}\frac{2}{2^k}.
\end{align*}
By definition of $k_0$, we finally get $\|\sum_{n\in F}u_{-n}\|\leq 2\varepsilon$. This shows the unconditional convergence of $\sum_{n\leq 0}u_n$.

Again by the triangle inequality and \eqref{existenceLawFrechetSpaceFHCineq1}, we also have
\begin{align*}
\left\|\sum_{n\in F}u_{n}\right\|
&=\bigg\|\sum_{n\in F}\sum_{k\geq 1}a_kS^{n_k+n}(x_k)\bigg\|\\
&\leq\sum_{k=1}^{k_0}\left\|\sum_{n\in F}a_kS^{n_k+n}(x_k)\right\|
+\sum_{k\geq k_0+1}\left\|\sum_{n\in F}a_kS^{n_k+n}(x_k)\right\|\\
&\leq\sum_{k=1}^{k_0}\left\|\sum_{n\in F}a_kS^{n_k+n}(x_k)\right\|
+\sum_{k\geq k_0+1}\frac{1}{2^k}.
\end{align*}
As before, the first term is smaller than $\varepsilon$ since $\min F\geq N_k$ for each $1\leq k\leq k_0$, and the second term is smaller than $\varepsilon$ by definition of $k_0$. This shows the unconditional convergence of $\sum_{n\geq 0}u_n$.
\end{proof}

\paragraph*{Acknowledgements.}
The author would like to thank Karl Grosse-Erdmann for his valuable advice and numerous readings of the manuscript. He is also grateful to Antoni L\'{o}pez-Mart\'{\i}nez for pointing out that unconditional convergence is not necessarily equivalent to bounded multiplier convergence in arbitrary F-spaces.

\bibliographystyle{plain}
\bibliography{amain}

\begin{thebibliography}{10}

\bibitem{BayartGrivaux2006}
Fr\'{e}d\'{e}ric Bayart and Sophie Grivaux.
\newblock Frequently hypercyclic operators.
\newblock {\em Trans. Amer. Math. Soc.}, 358(11):5083--5117, 2006.

\bibitem{BayartMatheron2009}
Fr\'{e}d\'{e}ric Bayart and \'{E}tienne Matheron.
\newblock {\em Dynamics of Linear Operators}.
\newblock Cambridge University Press, 2009.

\bibitem{BayartMatheron2016}
Fr\'{e}d\'{e}ric Bayart and \'{E}tienne Matheron.
\newblock Mixing operators and small subsets of the circle.
\newblock {\em J. Reine Angew. Math.}, 715:75--123, 2016.

\bibitem{BonillaGrosse-Erdmann2007}
A.~Bonilla and K.-G. Grosse-Erdmann.
\newblock Frequently hypercyclic operators and vectors.
\newblock {\em Ergodic Theory Dynam. Systems}, 27(2):383--404, 2007.

\bibitem{Coudene2016}
Yves Coudène.
\newblock {\em Ergodic Theory and Dynamical Systems}.
\newblock Springer-Verlag, 2016.

\bibitem{DajaniKalle2021}
Karma Dajani and Charlene Kalle.
\newblock {\em A First Course in Ergodic Theory}.
\newblock Chapman and Hall/CRC, 2016.

\bibitem{FonsecaLeoni2007}
Irene Fonseca and Giovanni Leoni.
\newblock {\em Modern Methods in the Calculus of Variations: $L^p$ spaces}.
\newblock Springer-Verlag, 2007.

\bibitem{Grivaux2006}
Sophie Grivaux.
\newblock A probabilistic version of the frequent hypercyclicity criterion.
\newblock {\em Studia Math.}, 176(3):279--290, 2006.

\bibitem{Grosse-ErdmannManguillot2011}
Karl-G. Grosse-Erdmann and Alfred~Peris Manguillot.
\newblock {\em Linear Chaos}.
\newblock Springer, 2011.

\bibitem{HytonenNeervenVeraarWeis2017}
Tuomas Hytönen, Jan van Neerven, Mark Veraar, and Lutz Weis.
\newblock {\em Analysis in Banach spaces}, volume~2.
\newblock Springer, 2017.

\bibitem{Kahane1960}
J.-P. Kahane.
\newblock Propri\'{e}t\'{e}s locales des fonctions \`a s\'{e}ries de {F}ourier
  al\'{e}atoires.
\newblock {\em Studia Math.}, 19:1--25, 1960.

\bibitem{KaltonPeckRoberts1984}
N.J. Kalton, N.T. Peck, and James~W. Roberts.
\newblock {\em An F-space sampler}.
\newblock Cambridge University Press, 1984.

\bibitem{KamthanGupta1981}
P.~K. Kamthan and Manjul Gupta.
\newblock {\em Sequence spaces and series}, volume~65 of {\em Lecture Notes in
  Pure and Applied Mathematics}.
\newblock Marcel Dekker, Inc., New York, 1981.

\bibitem{LiQueffelec2004}
Daniel Li and Hervé Queffélec.
\newblock {\em Introduction à l'étude des espaces de Banach}.
\newblock Société Mathématique de France, 2004.

\bibitem{LopesMessaoudiStadlbauerVargas2021}
Artur~O. Lopes, Ali Messaoudi, Manuel Stadlbauer, and Victor Vargas.
\newblock Invariant probabilities for discrete time linear dynamics via
  thermodynamic formalism.
\newblock {\em Nonlinearity}, 34(12):8359--8391, 2021.

\bibitem{MouzeMunnier2014}
A.~Mouze and V.~Munnier.
\newblock On random frequent universality.
\newblock {\em J. Math. Anal. Appl.}, 412(2):685--696, 2014.

\bibitem{MouzeMunnier2021}
Augustin Mouze and Vincent Munnier.
\newblock Frequent hypercyclicity of random holomorphic functions for {T}aylor
  shifts and optimal growth.
\newblock {\em J. Anal. Math.}, 143(2):615--637, 2021.

\bibitem{Murillo-ArcilaPeris2013}
M.~Murillo-Arcila and A.~Peris.
\newblock Strong mixing measures for linear operators and frequent
  hypercyclicity.
\newblock {\em J. Math. Anal. Appl.}, 398(2):462--465, 2013.

\bibitem{Nikula2014}
Miika Nikula.
\newblock Frequent hypercyclicity of random entire functions for the
  differentiation operator.
\newblock {\em Complex Anal. Oper. Theory}, 8(7):1455--1474, 2014.

\bibitem{Rudin1987}
Walter Rudin.
\newblock {\em Real and Complex Analysis}.
\newblock McGraw-Hill, third edition, 1987.

\bibitem{Walters1982}
Peter Walters.
\newblock {\em An Introduction to Ergodic Theory}.
\newblock Springer, 1982.

\end{thebibliography}
\nocite{*}

Département de Mathématique, Université de Mons, 20 Place du Parc, 7000 Mons, Belgium

E-mail address: kevin.agneessens@umons.ac.be

\end{document}